\documentclass[letter,11pt]{amsart}
\usepackage[toc,page]{appendix}
\usepackage{amssymb} 
\usepackage{bbm}
\usepackage{graphicx}
\usepackage{color}
\usepackage[font=footnotesize]{caption}
\usepackage{mathtools}
\usepackage{pinlabel}
\usepackage{graphicx}
\graphicspath{ {images/} }
\usepackage{subfig}
\usepackage{comment}
\usepackage{hyperref}

\usepackage{float}
\usepackage{calc}
\def\thetitle{{ . }}
\hypersetup{
	pdftitle=  \thetitle,
	pdfauthor=  {}  %Author name
}

\newtheorem{THM}{Theorem}[section]

\newtheorem*{THM*}{Theorem}
\newtheorem{LEM}[THM]{Lemma}

\theoremstyle{remark}

\usepackage{tikz}
\usetikzlibrary{calc,decorations.markings}
\usetikzlibrary{shapes,snakes}

\theoremstyle{definition}
\newtheorem*{defn*}{Definition}

\newcommand\R{\mathbb{R}}

\newcommand\Z{\mathbb{Z}}

\newcommand{\dist}{\mathrm{dist}}

\newcommand\Mod{\operatorname{Mod}}

\newcommand\PSL{\operatorname{PSL}}

\begin{document}
\title{On translation lengths of pseudo-Anosov maps on the curve graph}
\author{Hyungryul Baik}
\address{Department of Mathematical Sciences, KAIST,  
	291 Daehak-ro, Yuseong-gu, Daejeon 34141, South Korea }
\email{hrbaik@kaist.ac.kr}

\author{Changsub Kim}
\address{Department of Mathematical Sciences, KAIST,  
	291 Daehak-ro, Yuseong-gu, Daejeon 34141, South Korea }
\email{kcs55505@kaist.ac.kr}
              
\maketitle
%\keywords{}
\begin{abstract}
    We show that a pseudo-Anosov map constructed as a product of the large power of Dehn twists of two filling curves always has a geodesic axis on the curve graph of the surface.
    We also obtain estimates of the stable translation length of a pseudo-Anosov map, when two filling curves are replaced by multicurves. 
    Three main applications of our theorem are the following: (a) determining which word realizes the minimal translation length on the curve graph within a specific class of words, (b) giving a new class of pseudo-Anosov maps optimizing the ratio of stable translation lengths on the curve graph to that on Teichm{\" u}ller space, (c) giving a partial answer of how much powers will be needed for Dehn twists to generate right-angled Artin subgroup of the mapping class group.
\end{abstract}

\section{Introduction}
    Let $S=S_{g,n}$ be a connected orientable surface with genus $g$ and $n$ punctures.
    The mapping class group of the surface $S$, written as $\Mod{S}$, is the group of orientation-preserving homeomorphisms of $S$ up to isotopy.
    For a surface $S$, letting each isotopy class of simple closed curve as vertex, and connecting edge between two distinct vertices when they can be represented by pair of disjoint curves, we obtain the curve graph $\mathcal{C}(S)$.
    By assigning each edge length $1$, $\mathcal{C}(S)$ is a $\delta$-hyperbolic space \cite{masur1999geometry} and the natural action of $\Mod(S)$ on curve graph is an isometric action. 

    The asymptotic translation length(also known as stable translation length) of a mapping class $f\in \Mod(S)$ on the curve graph is defined as
  \[
   l_{\mathcal{C}}(f)=\lim_{n\rightarrow \infty}\frac{d_{\mathcal{C}}(\alpha,f^n(\alpha))}{n},
  \]
    where $\alpha$ is any vertex of $\mathcal{C}(S)$. 
    The triangle inequality guarantees that the definition of $l_{\mathcal{C}}(f)$ does not depend on the choice of $\alpha$. 
    Masur-Minsky showed that $l_{\mathcal{C}}(f)>0$ if and only if $f$ is a pseudo-Anosov element \cite{masur2000geometry}.
    Furthermore, Bowditch showed that $l_{\mathcal{C}}(f)$ is always a rational number by showing that there exists $n$ depending only on the type of surface such that $f^n$ preserves a geodesic axis \cite{bowditch2008tight}. In other words, the asymptotic translation of a pseudo-Anosov map on the curve graph is always a rational number with a bounded denominator. 

    While there have been many research works on the asymptotic translation length of pseudo-Anosov maps on the curve graph, there are only a few known results about which pseudo-Anosov maps have integer asymptotic translation lengths. In the case of a torus, authors together with Kwak and Shin proved that for any Anosov map on a torus, there is always a geodesic axis on the curve graph of torus \cite{baik2021translation}. Thus in this case the asymptotic translation length agrees with the translation length, and they are always positive integers. In the non-sporadic case, Watanabe in \cite{watanabe2020pseudo} obtained certain sufficient conditions for a pseudo-Anosov map to have an integer asymptotic translation length. 
   
     Here we construct another class of pseudo-Anosov maps that preserve geodesic axes and calculate the value of the translation lengths explicitly from our construction. Our main theorem is the following:

    \begin{THM*}[Theorem \ref{THM:main}]
	Let $\alpha, \beta$ be two filling curves with $d(\alpha,\beta)=l\geq 3$. Pick a geodesic path $\alpha, v_1,\cdots v_{l-1}, \beta$ in $\mathcal{C}(S)$ and let $f=T_a^{a_1} T_b^{b_1}\cdots T_a ^{a_n} T_b^{b_n}$. Then there exists a universal positive constant $M$ such that the following holds: 
	Assume all $|a_i|, |b_i| >2 M$. Let $p$ be a length $(l-2)$ geodesic path $(v_i)$, and $f_i$ be the map obtained by taking the first $i$ syllables in $f$. 
	Then concatenated path $p$, $f_1 p$,$\cdots$,$f_{2n} p=f p$ is a geodesic. 
    Furthermore, $f$ is a pseudo-Anosov map such that its geodesic axis is concatenated path's orbit by $f$, and its stable translation length is given by $l_\mathcal{C}(f)= 2n(l-2)$. 
\end{THM*}

    Gadre, Hironaka, Kent, and Leininger showed that the minimal value of the ratio between the asymptotic translation length on the curve graph and the translation length on the Teichm\"uller space is asymptotic to $\log |\chi(S)|$ \cite{gadre2013lipschitz}.
    Aougab-Taylor gave a construction of ratio optimizers which is a class of pseudo-Anosov maps that realize asymptotic minimal \cite{aougab2015pseudo}.
    As an application of the first theorem, we give a larger class of ratio optimizers.

    Generalizing the first theorem, we also give bounds on the asymptotic translation lengths of pseudo-Anosov maps when the number of curves used in the construction is more than two.
    We first deal with the case where a map is a product of the powers of Dehn twists, with each neighboring pair of curves in the map filling the surface. We compute how large the exponent, should be for each Dehn twist to give a bound of the asymptotic translation lengths.
    
    Two well-known constructions of pseudo-Anosov maps come from two filling multicurves, namely Thurston's and Penner's.
    In the case of two filling multicuves, we can also give bounds on the asymptotic translation length of a product of Dehn twists, where the exponents in the powers of Dehn twists satisfy a certain condition.
    This allows us to calculate the asymptotic length of pseudo-Anosov maps on a curve graph in Thurston's construction and Penner's construction, as long as we take high enough powers of Dehn twists. 
    Furthermore, we figure out how large the exponent of powers should be in order to give similar bounds when the number of multicurves is more than two.
    This bound can be used to find words that give minimal translation length in a specific class of words.

    It is proven by Koberda that for a finite set of curves $C$, there exists $K$ such that for $n>K$, $\{ T_\gamma ^n \in C \}$ generates right-angled Artin groups \cite{koberda2012right}. 
    In the same paper, it is asked if one can give an explicit estimate of $K$. It is also known that $K$ depends on the choice of curves in $C$. 
    As an application of our bounds on the asymptotic translation lengths on the curve graphs, we give an explicit estimate of $K$ to generate a right-angled Artin group under a specific condition on the configuration of the curves. 

\subsection*{Acknowledgement}
    We thank Dongryul M. Kim, KyeongRo Kim, Sanghoon Kwak, and Donggyun Seo for their helpful comments. Both authors were supported by National Research Foundation of Korea(NRF) grant funded by the Korea government(MSIT) (No. 2020R1C1C1A01006912). 

\section{Background}
\subsection{Subsurface Projection}
Masur and Minsky defined subsurface projection in \cite{masur2000geometry}, which is a map from the vertices of the curve graph of a surface to the power sets of vertices curve graph of a subsurface.
In the case where the subsurface is annular, subsurface projection sends the core of an annulus and curves not intersecting the core to the empty set.

\begin{THM}[Bounded geodesic image theorem, \cite{masur2000geometry}]\label{bdd}
	Assume for a subsurface $Y$ in $S$ and geodesic $g$ in $\mathcal{C}$, $g$'s all vertices have a nontrivial projection on the curve graph of $Y$. Then the whole projection of $g$ on $Y$'s curve graph has a diameter smaller than $M$. 
\end{THM}

It is proven by Webb that $M$ can be chosen as a universal constant that doesn't depend on the type of $S$ or $Y$ \cite{webb2015uniform}.
For a simple closed curve $\alpha$ on a surface, let's say annular subsurface around $\alpha$ as $Y$. We use the notation $d_\alpha=d_Y$ as the distance on the curve graph of subsurface $Y$.

Bounded geodesic image theorem implies that if three curve $\alpha, \beta, \gamma \in \mathcal{C}(S)$ satisfies $d_\alpha(\beta,\gamma)>M$, any geodesic path joining $\beta$ and $\gamma$ must intersect one-neighborhood of $\alpha$ in $\mathcal{C}(S)$.

We will use the bounded geodesic image theorem to restrict the geodesic path between two points. 
To enlarge the distance on the curve graph of a subsurface, we need the following lemma. 

\begin{LEM}[\cite{masur2000geometry} ]\label{dehn}
    For $\alpha \in \mathcal{C}_0(S)$ and $\beta \in \mathcal{C}_0(\alpha)$, following holds.
 $$d_\alpha (\beta, T_\alpha^n (\beta))=|n|+2.$$
\end{LEM}	

\subsection{Pseudo-Anosov Construction}
We call a set of curves as a multicurve if it is made of non-intersecting essential simple closed curves.
For two multicurves $A=\{\alpha_1,\cdots,\alpha_m\}$ and $B=\{\beta_1,\cdots, \beta_n\}$ on surface $S$, we say $A$ and $B$ fill $S$ if complement $S/(A\cup B)$ contains only disks and once punctured disks components.
Two curves $\alpha$ and $\beta$ satisfy $d_\mathcal{C}(\alpha, \beta)\geq 3 $ if and only if when $\alpha$ and $\beta$ are filling curves.
This can be shown by that if $S/(\alpha \cup \beta)$ contains a component that has higher complexity, a curve can be drawn on it avoiding $\alpha$ and $\beta$, which makes $d_\mathcal{C}(\alpha, \beta)\leq 2 $.

There are two well know pseudo-Anosov constructions using Dehn twists of two filling multicurves, which are by Thurston and Penner \cite{thurston1988geometry}, \cite{penner1988construction}.
Further details and explanations of those constructions can be also found in Chapter 14.1 in \cite{farb2012primer}.
For a multicurve $A=\{\alpha_1,\cdots,\alpha_m\}$, let's denote $T_A=T_{\alpha_1} \cdots T_{\alpha_m}$.

\begin{THM}[\cite{thurston1988geometry}]\label{thurston}
    Assume $A=\{\alpha_1,\cdots,\alpha_m\}$ and $B=\{\beta_1,\cdots, \beta_n\}$ be two multicurves that fill a surface $S$. 
	Let $N$ be the matrix by $N_{i,j}=i(\alpha_i,\beta_j)$ and $\mu$ be the largest real eigenvalue of $NN^T$. Then there exists a representation $\rho :\langle T_A,T_B\rangle  \to  \PSL{(2,\R)}$ given by
	\[
	T_A \mapsto 
	\begin{bmatrix} 
		1 & \sqrt{\mu} 
		\\ 0 & 1	
	\end{bmatrix} \quad \textrm{and} \quad T_B \mapsto 
	\begin{bmatrix} 	
		1 & 0 \\
		-\sqrt{\mu} & 1 	
	\end{bmatrix}.
	\]
 Where $f\in \langle T_A, T_B\rangle$ is a pseudo-Anosov map if and only if $\rho(f)$ is hyperbolic, and if $f$ is a pseudo-Anosov map, the stretch factor of $f$ is equal to $\rho(f)$'s largest real eigenvalue.
\end{THM}

\begin{THM}[\cite{penner1988construction}]
    Assume $A=\{\alpha_1,\cdots,\alpha_m\}$ and $B=\{\beta_1,\cdots, \beta_n\}$ be two multicurves that fill a surface $S$. If $f$ is a word made from the product of positive Dehn twists about $\alpha_i $ and negative Dehn twists about $\beta_j$, where all $\alpha_i, \beta_j$ are used at least once, then $f$ is a pseudo-Anosov map.
\end{THM}

Mangahas suggested a new recipe for pseudo-Anosov maps using pure reducible mapping classes \cite{mangahas2013recipe}.
To prove the map is pseudo-Anosov, Mangahas used the following Lemma to show map's stable translation length is positive.
We will also use a slightly altered version of Lemma which is restricted to the case of annular subsurfaces to construct classes of pseudo-Anosov maps and to calculate stable translation length.
Let $M$ be the universal constant from the bounded geodesic image theorem.

\begin{LEM}[\cite{mangahas2013recipe}]\label{geo} 
    For set of curves $\{\alpha_1,\cdots,\alpha_n\}$, let $X_i=N_1(\alpha_i)$, $Y_i$ be annular subsurface around $\alpha_i$. 
    Assume $d_{\mathcal{C}}(\alpha_i,\alpha_{i+1}) \geq 3$ and $d_{\alpha_i}(\alpha_{i-1},\alpha_{i+1})>2M+2$ for all $i$. 
    Then for any $w_i \in X_i$ and $w_{i+k} \in X_{i+k}$, the geodesic $[w_i,w_{i+k}]\cap X_j \neq \emptyset$ for $i \leq j \leq i+k$. Also, the $X_j$ are pairwise disjoint.  
    Furthermore there exist a geodesic $[w_i,w_{i+k}]$ where elements from $X_i, X_{i+1} \cdots, X_{i+k}$ appear in same order in geodesic. 
\end{LEM}

\section{Translation length on curve graph}

Using Mangahas' lemma, Aougab-Taylor gave a lower bound for the asymptotic translation length of pseudo anosov map that is made from large enough power of two filling curves' Dehn twists \cite{aougab2015pseudo}. Here we extend that idea to give the exact asymptotic translation length by showing that it preserves a geodesic.

Let's say Dehn twists around curve $\alpha$ and $\beta$ as $T_a$ and $T_b$, and let $M$ as the universal constant from the bounded geodesic image theorem. 

\begin{THM}\label{THM:main}
	Let $\alpha, \beta$ be two filling curves with $d(\alpha,\beta)=l \geq 3$. 
    Pick a geodesic path $\alpha, v_1,\cdots v_{l-1}, \beta $ and let $f=T_a^{a_1} T_b^{b_1}\cdots T_a ^{a_n} T_b^{b_n}$. 
	Assume all $|a_i|, |b_i| >2 M$. Let be a $p$ length $(l-2)$ geodesic path $(v_i)$, and $f_i$ be the map obtained by taking the first $i$-th syllable in $f$. 
	Then concatenated path $p$, $f_1 p$,$\cdots$,$f_{2n} p=f p$ is geodesic. Furthermore, $f$ is a pseudo-Anosov map such that the geodesic axis is concatenated path's orbit by $f$, and stable translation length is given by $l_\mathcal{C}(f)= 2n(l-2)$. 
	
\end{THM}
 
\begin{proof}
	\begin{figure}[h]
	\centering
	\includegraphics[width = 0.7\textwidth]{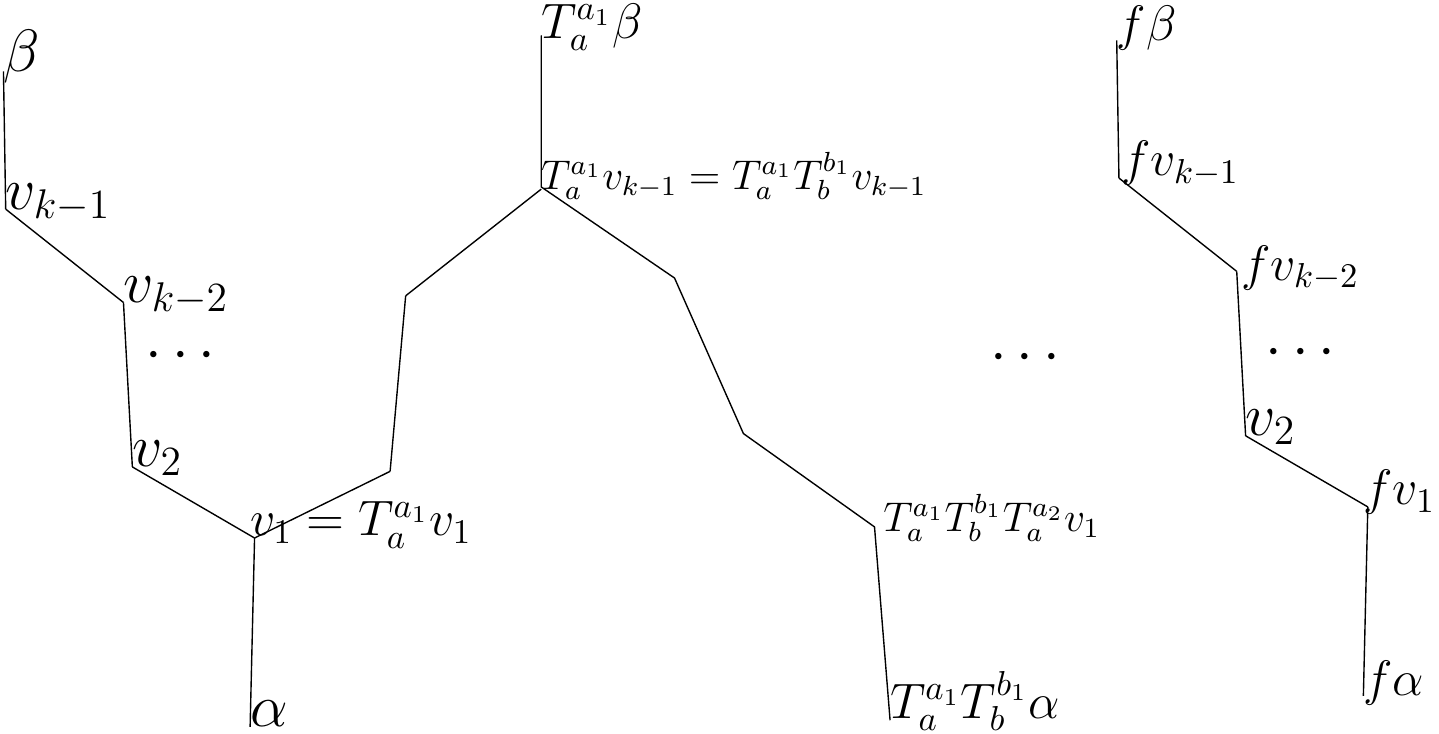}
	\caption[]{Geodesic axis of map $f$}
\end{figure}

	Let $C_A=N_1(\alpha)$, $C_B=N_1(\beta)$, and let $X_0=C_B$, $X_1=C_A$, $X_{2i}=f_{2i} C_B$, $X_{2i+1}=f_{2i+1} C_A$.
    Now shortest path between $X_i$ and $X_{i+1}$ will be length $l-2$, since it is isomorphic image of $C_A$ and $C_B$ by $f_{i}$.
	This path can be realized as $f_i p =(f_i v_1,\cdots, f_i v_{l-1})$.
	By canceling $f_{2i-1}$, we get following by Lemma \ref{dehn}:
    $$d_{f_{2i} \beta}(f_{2i-1} \alpha, f_{2i+1} \alpha)=d_{ \beta}( \alpha, T_\beta^{b_i} \alpha)=|b_i|+2>2M+2.$$
    Similarly $d_{f_{2i+1} \alpha}(f_{2i} \beta, f_{2i+2} \beta)>2M+2 $, thus conditions of Lemma \ref{geo} are satisfied.
    Let's pick a geodesic path between $v_1$ and $f v_1$. 	
    Then by Lemma \ref{geo} this path must meet all $X_i$ between them. 
    It must contain $2n$ paths of $X_i$ and $X_{i+1}$ and they should be all disjoint, so has a length of at least $2n(l-2)$.
    Concatenated path of $f_1 p$ to $f_{2n} p$ realizes this path, so we have $d(v_1, f v_1)=2n(l-2)$. 
    By our setting, same logic holds for $f^m$ for any positive $m$, thus it makes $d(v_1, f^m v_1)=2mn(l-2)$, which proves rest.
\end{proof}

This can be also applied to give bound to pseudo-Anosov maps made with more than two Dehn twists. 

\begin{THM} \label{manycurves}
		Consider curves $\alpha_1,\alpha_2,\cdots,\alpha_k$ which satisfy $d_{\mathcal{C}}(\alpha_i,\alpha_{i+1})\geq3$, $\alpha_{k}=\alpha_0$ and $\alpha_{k+1}=\alpha_1$. 
	For $f=T_{\alpha_1}^{n_1}T_{\alpha_1}^{n_2}\cdots T_{\alpha_k}^{n_k}$ 	with $|n_i| > 2M+2+d_{\alpha_i}(\alpha_{i-1},\alpha_{i+1})$, its translation length is bounded as follows:
	\begin{align*} 
		\Sigma_{i=1}^k d_{\mathcal{C}} (\alpha_i,\alpha_{i+1})-2k  \leq l_{\mathcal{C}}(f)  \leq \Sigma_{i=1}^k d_{\mathcal{C}}(\alpha_i,\alpha_{i+1}) ,
	\end{align*}
	 thus $f$ is a pseudo-Anosov map.
\end{THM}

\begin{proof}
    Let $f_i = T_{\alpha_1}^{n_1}T_{\alpha_1}^{n_2}\cdots T_{\alpha_k}^{n_i}$, $X_i=f_i N_1(\alpha_i)$ with $Y_i$ be annular subsurface that has core $\beta _i = f_i \alpha_i $ where index is given $\mod k$. 
     Now $d_{\beta_i}(\beta_{i-1},\beta_{i+1})=d_{f_i \alpha_i}(f_{i-1} \alpha_{i-1}, f_{i+1} \alpha_{i+1}) $. 
    Since $\Mod{S}$ preserves relation between curves on $\mathcal{C}(S)$, deleting $f_{i-1}$ gives $d_{f_i \alpha_i}(f_{i-1} \alpha_{i-1}, f_{i+1} \alpha_{i+1}) = d_{\alpha_i}(\alpha_{i-1},T_i^{n_i}\alpha_{i+1})$. 
    From triangle inequality we get $$d_{\alpha_i}(\alpha_{i-1},T_i^{n_i}\alpha_{i+1}) \geq d_{\alpha_i}(\alpha_{i+1},T_i^{n_i}\alpha_{i+1}) - d_{\alpha_i}(\alpha_{i-1},\alpha_{i+1})>2M+2.$$
    Thus it satisfies the condition of Lemma \ref{geo}.
    Now pick any $\omega_0 \in X_0=N_1(\alpha_{0})$. Then we can see geodesic path from $\omega_{0}$ to $f^t \omega_{0}$ have intersection wit all $X_i$ for all $0<i<tk$. Thus we get
    $$ t(\Sigma_{i=1}^k d_{\mathcal{C}} (\alpha_i,\alpha_{i+1})-2k) \leq d_{\mathcal{C}}(\omega_{0},f^t \omega_{0}) \leq t(\Sigma_{i=1}^k d_{\mathcal{C}}(\alpha_i,\alpha_{i+1})). $$
    Dividing into $t$ and sending it to infinity gives us the desired result.
\end{proof}
Here we can see that filling two single curves is a special case of the theorem, where the lower bound coincides with the exact translation length.

Though it may not preserve a geodesic like in two curve case, we can still give bounds to stable translation lengths of two well-known constructions of pseudo-Anosov maps, Thurston's construction and Penner's construction, under some conditions. 
To track the length on the curve graph of the subsurface in case of multicurve's Dehn twist, we need the following lemma from \cite{watanabe2020pseudo}.

\begin{LEM}\label{multidehn}\cite{watanabe2020pseudo}
	For multicurve $A=\{\alpha_1,\cdots,\alpha_n\}$. Then for any curve $\gamma$ intersecting $\alpha_1$ transversely let $f= T_{\alpha_1}^{a_1}\cdots T_{\alpha_n}^{a_n}$
	$$d_{\alpha_1}(\gamma ,f(\gamma )) \geq |a_1|$$
	 
\end{LEM}

On the curve graph of an annular subsurface, while $n$ times of Dehn twist about the core of an annulus sends an image to distance $|n|+2$ place, Dehn twists about curves that don't intersect the core have a minor effect.
For $A, B \subset \mathcal{C}(S)$, let's define $\dist(A,B)=\inf \{ d_{\mathcal{C}}(a,b)|a\in A, b\in B\}$.

\begin{THM}\label{THM:multicurve}
	Let filling multicurves $A=\{\alpha_1,\cdots,\alpha_m\}$, $B=\{\beta_1,\cdots,\beta_n\}$ satisfy $\dist(A,B)=l \geq 3 $. 
	Let $f=f_1\cdots f_{2k}$, where $f_{2i-1}, f_{2i}$ are a products of Dehn twists by curves in $A$, $B$ respectively and each $f_i$ contains at least one Dehn twist $T_{\gamma}^t$ with $|t|>2M+3$.
    Then	
    $$2kl-4k \leq  l_{\mathcal{C}}(f)  \leq 2kl,$$
    thus $f$ is a pseudo-Anosov map.
\end{THM}

\begin{proof}
    Let $F_i=f_1\cdots f_i$, and $X_i=F_i N_1(\gamma_i)$, where $\gamma_i$ is a curve from $f_i$ which Dehn twist's exponent's absolute value is largest. 
    Here $\gamma_i$ is either element of $A$ or $B$, thus $A\subset N_1(\gamma_i)$ or $B\subset N_1(\gamma_i)$ since $A$ and $B$ are multicurves. 
    From lemma \ref{multidehn} we get $d_{\gamma_i}(\gamma_{i+1},f_i(\gamma_{i+1}))>2M+3$, thus $d_{\gamma_i}(\gamma_{i-1},f_i(\gamma_{i+1}))>2M+2$. 
    Letting $Y_i$ as an annular subsurface around core $F_i \gamma_i $, $X_i, Y_i$ satisfies conditions of lemma \ref{geo}. 
    Any geodesic connecting $N_1(\gamma_1)$ to $f N_1(\gamma_1)$ must have nonempty intersection with $X_i$. 
    On our geodesic, distance between intersection of $X_{i}$ and $X_{i+1}$ is bounded between $l-2$ and $l$. 
    Thus we get following
    $$2kl-4k \leq \dist(X_1,f X_1)  \leq 2kl. $$
    All our logic still holds for $f^t$, thus we can expand the situation to
    $$t(2kl-4k) \leq \dist(X_1,f^t X_1)  \leq 2tkl. $$
    Dividing all sides by $t$ gives us the desired result.
\end{proof}
    
\begin{THM}\label{THM:multicurves}
	Let multicurves $C_1\cdots C_n$ satisfy $\dist(C_i,C_{i+1} )\geq 3 $, $C_n=C_0$ and $C_{n+1}=C_1$. 
	Let $f=f_1\cdots f_{n}$ where $f_i$ is product of Dehn twists by curve in $C_i$ where each $f_i$ contains at least one Dehn twist $T_{\gamma}^t$ with $|t|>2M+3+d_{\gamma}(C_{i-1},C_{i+1})$. 
    Then following holds
	\begin{align*} 
	\Sigma_{i=1}^n \dist (C_i,C_{i+1})-2n  \leq l_{\mathcal{C}}(f)  \leq \Sigma_{i=1}^n \dist(C_i,C_{i+1}) ,
\end{align*}
thus $f$ is a pseudo-Anosov map.
\end{THM}
\begin{proof}
    Let $F_i=f_1\cdots f_i$, and $X_i=F_i N_1(\gamma_i)$, where $\gamma_i$ is a curve from $f_i$ which Dehn twist's exponent's absolute value is largest. 
    From lemma \ref{multidehn} we get $d_{\gamma_i}(\gamma_{i+1},f_i(\gamma_{i+1}))>2M+3+d_{\gamma}(C_{i-1},C_{i+1})$. Then by canceling $F_{i-1}$ and using triangle inequality, we get following: 
    \begin{align*}
     d_{F_i\gamma_i}(F_{i-1}\gamma_{i-1},F_{i+1}\gamma_{i+1}))&=d_{\gamma_i}(\gamma_{i-1},f_i(\gamma_{i+1})) \\
    & =d_{\gamma_i}((\gamma_{i+1},f_i(\gamma_{i+1}))-d_{\gamma_i}(\gamma_{i-1},\gamma_{i+1})  \\
    & > 2M+2+d_{\gamma}(C_{i-1},C_{i+1})-d_{\gamma_i}(\gamma_{i-1},\gamma_{i+1})  \\
    & \geq 2M+2.
    \end{align*}
    
    Thus letting $Y_i$ as annulur subsurface around core $F_i \gamma_i $, our $X_i, Y_i$ satisfy conditions of lemma \ref{geo}. 
    $$\Sigma_{i=1}^n \dist (C_i,C_{i+1})-2n  \leq d_{\mathcal{C}}(X_1, f X_1)  \leq \Sigma_{i=1}^n \dist(C_i,C_{i+1})  $$
    All our condition still holds for $f^k$, thus changing to $f^k$ and dividing all the sides by $k$ gives us desired result.
\end{proof}

\section{Applications}
\subsection{Minimal word}
One direction of understanding translation length is to find the smallest translation length element among the ones that use the same number of each Dehn twist. We give a partial answer to this question.
\begin{THM}
	Let $A=\{a_1,\cdots,a_m\}$ $B=\{b_1,\cdots,b_n\} $ be filling multicurves.
    For pseudo-Anosov $f$ that satisfies the condition in the Theorem \ref{THM:multicurve}, let $r_i, s_j $ be a total amount of powers of Dehn twist in $a_i, b_j$ respectively. Let $f'=T_{a_1}^{r_1}\cdots T_{a_n}^{r_m}T_{b_1}^{s_1}\cdots T_{b_m}^{s_n}$. Then
    
	$$ l_\mathcal{C}(f) \geq l_\mathcal{C}(f') .$$
\end{THM}
\begin{proof}
    Let $\dist(A,B)=l$, and $k$ be number of interchange to $A$ to $B$ in $f$. Then by Theorem \ref{THM:multicurve}, $l_{\mathcal{C}}(f')<2l$ and $2kl-4k \leq  l_{\mathcal{C}}(f)  $.  Thus it is guaranteed that if $k \geq 2$, $ l_\mathcal{C}(f) > l_\mathcal{C}(f') .$ On the other hand if $k=1$, it means that $f$ is actually conjugation of $f'$ since in each multicurves Dehn twists commute. Conjugating element will have same translation length, thus $ l_\mathcal{C}(f) = l_\mathcal{C}(f') .$ 
\end{proof}

\subsection{Ratio Optimizer}
Let $l_\mathcal{T}(f) $ be an asymptotic translation length of $f$ on Teichm\"uller space.  Gadre-Hironaka-Kent-Leininger proved that minimal value of ratio $\tau(f)=\frac{l_\mathcal{T}(f)}{l_\mathcal{C}(f)}$ is asymptotic to $\log (|\chi(S)|)$. \cite{gadre2013lipschitz}. 
Aougab-Taylor used Mangahas' Lemma \ref{geo} to find a large class of pseudo-Anosov maps that optimize this ratio \cite{aougab2015pseudo}. 
Here we expand the class of ratio optimizers using another result from Aougab-Taylor.
\begin{THM}[\cite{aougab2014small}]\label{THM:minint}
	Let $\omega(g,p)=3g+p-4$. Then for any $g,p$ with $\omega(g,p)>0$, there exists an infinite geodesic ray $\gamma={v_0,v_1,v_2\cdots}$ such that 
	
	$$i(v_i,v_j)=O(\omega^{|j-i|-2}) $$
\end{THM}

\begin{LEM}[\cite{aougab2015pseudo}]\label{lem:trace}
   Let $a=\begin{pmatrix}  1 & t\\ 
  0 & 1 \end{pmatrix}$, $b=\begin{pmatrix}  1 & 0\\ 
  t & 1 \end{pmatrix}$ and let $M=a^{a_1}b^{b_1}\cdots a^{a_n}b^{b_n}$ where $a_i, b_j \in \{1,-1\}$ and $t>1$. Then $\mathrm{trace} (M)\leq (2t)^{2n}$.
\end{LEM}

\begin{THM}
	Let $\gamma={v_0,v_1,v_2\cdots}$ be a geodesic ray in Theorem \ref{THM:minint}, and let $T_i=T_{v_i}$ as Dehn twist around $v_i$. 
    Then for any $i,j$ such that $|i-j|\geq 3$, let $a=T_{i}^{2M+1}$, $b=T_{j}^{2M+1}$ and $f=a^{a_1}b^{b_1}\cdots a^{a_n}b^{b_n}$ where $a_r, b_s \in \{1,-1\}$. Then we have the following:  
		$$\tau(f)=\frac{l_\mathcal{T}(f)}{l_\mathcal{C}(f)} =O(\log(\omega)) .$$
\end{THM}
\begin{proof}
    Let's map $a\mapsto \begin{pmatrix}  1 & t\\ 
  0 & 1 \end{pmatrix}$ and  $b\mapsto \begin{pmatrix}  1 & 0\\ 
  t & 1 \end{pmatrix}$, where $t=i(v_i,v_j)\cdot(2M+1)$.
  Since our $f$ is Thurston's construction made using two filling single curves, we can see that the stretch factor of $f$ coincides with the largest eigenvalue of matrix form by Theorem \ref{thurston}. 
  $a, b$ are mapped to matrix which has determinant $1$, thus $f$'s matrix form also has determinant $1$. Two eigenvalues are the multiplicative inverse of each other, and the trace is a sum of two. 
  Thus largest eigenvalue is smaller than the trace, and by Lemma \ref{lem:trace}, smaller than $(2t)^{2n}$. $l_\mathcal{T}(f)$ is logarithm of stretch factor, thus $l_\mathcal{T}(f)<2n \log (2t)$. We know that $l_\mathcal{C}(f)=2n(|j-i|-2)$ by Theorem \ref{THM:main} where we exactly calculated translation length. Combining these with \ref{THM:minint} gives us the desired result.
\end{proof}
\subsection{Right-Angled Artin Group}
Koberda proved that for a finite set of curves $C$, there exists $K$ such that for $n>K$, $\{ T_\gamma ^n \in C \}$ generates right-angled Artin groups \cite{koberda2012right}.
Koberda suggested a problem, how to decide the value of $K$.

While it is proven that the value of $K$ depends on the choice of $C$, in the general case of curve sets answers are given in \cite{seo2021powers}, \cite{runnels2021effective}.
On curves that have certain relations on the curve graph, we can answer Koberda's question using our theorems. 
Results in \cite{seo2021powers}, \cite{runnels2021effective} are bounded by using terms of intersection number of pair of curves.
Here we give answers in terms of the distance of two curves on subsurface projection on other curves, which is less or equal to the intersection number of curves.

\begin{THM}
    Consider curves $\alpha_1,\alpha_2,\cdots,\alpha_t$ which satisfy $d_{\mathcal{C}}(\alpha_i,\alpha_j)\geq3$ for any distinct $i,j$. 
	Let $n> 2M+2+\max {d_{\alpha_k}(\alpha_i,\alpha_j)}$. Then the set $\{ T_{\alpha_1} ^n,\cdots ,T_{\alpha_t} ^n \} $ generates the free group of rank $k$ .
	
\end{THM}
\begin{proof}
    Consider element in group generated from $\{ T_{\alpha_1} ^n,\cdots T_{\alpha_t} ^n \} $. 
    If it contains Dehn twists of more than two curves, consider its cyclically reduced form.
    Our $n$ is chosen to satisfy the condition in our theorem \ref{manycurves}. 
    A cyclically reduced element is a pseudo-Anosov map, thus original element is also a pseudo-Anosov map.
    This means any word written by more than one generator can not be trivial, so there can't exist any relation in the group.
    
\end{proof}

Using the same logic, in the multicurves cases, we can find the constant needed to guarantee that any word that contains Dehn twists from at least two multicurves will become a pseudo-Anosov map.
Thus that element can not be trivial, so there can't exist a relation in the group containing more than one multicurve.
Each set of Dehn twists on multicurve generates a free Abelian group, thus we get a right-angled Artin group which is a free product of free Abelian groups.
\begin{THM}
	Consider filling multicurves $A, B$ with $\dist(A,B)\geq 3$. For $n > 2M+3$, $\{ T_\gamma ^n |\gamma \in A\cup B \}$ generates $\Z^{|A|}*\Z^{|B|} $.
\end{THM}

\begin{THM}
	Let multicurves $C_1\cdots C_n$ satisfy $\dist(C_i,C_j )\geq 3 $. 
	For $n > 2M+3+\max \{d_\gamma(C_i,C_j)|\gamma \in \bigcup C_k\}$, 
	$\{ T_\gamma ^n |\gamma \in \bigcup C_k\} $ generates  $ \Z^{|C_1|}*\cdots*\Z^{|C_n|} $.
\end{THM}

\bibliographystyle{spmpsci}
\bibliography{refs}

\end{document}